\documentclass[final, 3p, 12pt]{elsarticle}
\usepackage{amsmath,hyperref}
\usepackage{amssymb}
\usepackage{amsthm,amsfonts}
\usepackage{float} %use [H] on \table to fix position
\usepackage{verbatim,xcolor}
\usepackage{algorithm}
\usepackage{algpseudocode}

\newtheorem{thm}{Theorem}

\newtheorem{crly}{Corollary}

\newtheorem{lma}{Lemma}
\theoremstyle{definition}
\newtheorem{dfn}{Definition}
\newtheorem{exx}{Example}
\newtheorem{fact}{Fact}

\newcommand{\blue}{\color{black}}
\newcommand{\red}{\color{red}}

\DeclareMathOperator{\m}{mult}
\DeclareMathOperator{\lcm}{lcm}

\DeclareMathOperator{\Stab}{Stab}
\DeclareMathOperator{\Orb}{Orb}

\newcommand\ellstar{{\ell^*}}

\newcommand{\D}{\mathbf{D}}
\newcommand{\mP}{\mathbf{P}}
\newcommand{\mQ}{\mathbf{Q}}
\newcommand{\A}{\mathbf{A}}
\newcommand{\T}{\mathbf{T}}
\newcommand{\I}{\mathbf{I}}
\newcommand{\JJ}{\mathbf{J}}
\newcommand{\cI}{\mathcal{I}_{\delta,G}}

\newcommand{\PP}{\boldsymbol{p}}

\newcommand{\C}{\boldsymbol{f}}
\newcommand{\R}{\boldsymbol{r}}
\newcommand{\Q}{\boldsymbol{q}}
\newcommand{\uuu}{\boldsymbol{u}}
\newcommand{\ttt}{\boldsymbol{t}}
\newcommand{\vv}{\boldsymbol{v}}

\newcommand{\xx}{\boldsymbol{x}}
\newcommand{\sols}{\boldsymbol{sols}}
\newcommand{\y}{\boldsymbol{y}}
\newcommand{\ZZZ}{\mathbb{Z}_{\ell_1}\times\cdots \times\mathbb{Z}_{\ell_r}}

\newcommand{\Z}{\mathbb{Z}}
\def\multiset#1#2{\ensuremath{\left(\kern-.3em\left(\genfrac{}{}{0pt}{}{#1}{#2}\right)\kern-.3em\right)}}

\begin{document}

%\begin{comment}%===============================

\begin{frontmatter}
\title{Decimation classes of nonnegative integer and fixed density vectors using multisets}

\author[FINDLAY]{Daniel M. Baczkowski}
\ead{baczkowski@findlay.edu}
\author[AFIT]{Dursun A.~Bulutoglu\corref{cor1}}
\ead{dursun.bulutoglu@gmail.com}
\address[FINDLAY]{Department of Mathematics, University of Findlay, Ohio 45840, USA}
\address[AFIT]{Department of Mathematics and Statistics, Air Force Institute of Technology,\\Wright-Patterson Air Force Base, Ohio 45433, USA}
%\journal{Journal of Algebraic Combinatorics}
\begin{abstract}
We describe how previously known methods for determining the number of decimation classes of density $\delta$ binary vectors can be extended to nonnegative integer vectors, where the vectors are indexed by a finite abelian group $G$ of size $\ell$ and exponent $\ellstar$ such that $\delta$ is relatively prime to $\ellstar$.  We extend the previously discovered theory of multipliers for arbitrary {\em subsets} of finite abelian groups, to arbitrary {\em multisubsets} of finite abelian groups.  
%We also extend previously known results on orbits under the action of the multiplier group, and finding the number of solutions of a potentially highly symmetric sum problem (SP).
Moreover, this developed theory provides information on the number of distinct translates fixed by each member of the multiplier group as well as sufficient conditions for each member of the multiplier group to be translate fixing.
\end{abstract}
\end{frontmatter}

%\end{comment}%===============================

\section{Introduction}

Theorems and algorithms were developed in~\cite{decimation} for counting the decimation classes of binary vectors with constant density. We closely follow this work done for binary vectors and show similar methods apply when replacing binary vectors with nonnegative integer vectors with constant density.  Many of the proofs from this work hold for the setting of nonnegative integer vectors.  Similar algorithms apply and some results are simplified within this more general context of nonnegative integer vectors.  
%Then we discuss how corresponding nonnegative integer vector results can be used for developing a theory for the compressions of Legendre pairs.

A {\em multiset} is a collection of elements where an element is allowed to occur more than once. Such an occurrence is not allowed for sets.  For a finite set $S=\{ r_1,\ldots,r_n \}$, let $I=\{ r_1^{\m(s_1)},\ldots,r_n^{\m(s_n)} \}$ denote the multiset with $s_i$ having multiplicity $\m(r_i)=\m_I(r_i)$ for each $i=1,\ldots,n$.  To prevent confusion we often write $\m_I(\cdot)$ to clarify the underlying multiset $I$.  The cardinality of a multiset is the sum of its multiplicities.  We will often let $\sigma$ denote the sum of all elements in a finite multiset; that is, $\sigma=\sum_{i\in I} i = \sum_{r\in S} \m_I(r)\,r$.
Let $x_i=\m_I(r_i)$ for each $i=1,\ldots,n$.  The number of all such possible multisets $I$ with cardinality $k$ is equal to the number of solutions to
\begin{equation}\label{2.1.1}
\sum_{i=1}^n x_i =k,\quad x_i\in\Z_{\ge 0},
\end{equation}
i.e., the number of nonnegative integer vectors whose components sum to $k$. It is well-known that the number of solutions to (\ref{2.1.1}) is
$$\binom{n+k-1}{k}.$$
Throughout this work we will use the notation
$$\multiset{n}{k} = \binom{n+k-1}{k}$$
for the number of multisets of cardinality $k$ with elements from a set $S$ with $|S|=n$. 

Throughout, we consider vectors indexed by a finite abelian group $G$ of order $\ell$.  
The density of a vector $\vv \in \{\Z_{\ge 0}\}^G$ is defined to be $\delta=\sum_{g \in G}v_{g}$. 
We use $|\cdot|$ to denote the order of a group or the cardinality of a multiset.  
 For an isomorphism $\Phi: G_1\rightarrow G_2$ and a multiset $I$ with elements from $G_1$, $\Phi(I)=\{\Phi(i)^{\m(i)}\mid i \in I\}$.
 We assume that $G$ is a finite abelian group that has the operation~$+$ unless specified otherwise.  
It is well known that such a group is isomorphic to $\ZZZ$ for some $\ell_1,\ell_2,\ldots,\ell_r \in \mathbb{Z}_{\geq 1}$. 
Again, we let $\mathbb{Z}_{\ell}^\times=\{j \in \mathbb{Z}_{\ell}\, |\, \gcd(j,\ell)=1\}$.  Let
$\ellstar$ denotes the \textit{exponent} of the group~$G\cong\ZZZ$, i.e., the smallest positive integer $n$ such that $ng=0$ for all $g \in G$. 
Any subgroup $H$ or $K$ of $\mathbb{Z}_\ellstar^\times$ is always written multiplicatively with its identity element equal to $1$.
A{ \em decimation} of a vector $\vv\in \{\Z_{\ge 0}\}^G$ by
 $j\in \mathbb{Z}_{\ellstar}^\times$, denoted by $d_j(\vv)$, is defined to be $\left(d_j(\vv)\right)_{g}=v_{j*g}$ for each $g\in G$.  It is easy to see for $G\cong\ZZZ$ that $\ellstar=\lcm(\ell_1,\ldots,\ell_r)$ and $|G|=\ell=\prod_{i=1}^r\ell_i$.

The \textit{necklace} of a vector $\vv\in\{\Z_{\ge 0}\}^G$ of length $\ell$ is the orbit of $\vv$ under circulant shifts, whereas the \textit{bracelet} of $\vv$ is the orbit of $\vv$ under circulant shifts and decimation by $-1$ (\textit{reversals})~\cite{Dokovic2015,Sawada2013}. Throughout the paper, we denote 
the necklace containing $\vv$ (the necklace of $\vv$) by $U_{\vv}$.
 The \textit{decimation class} of 
a vector is the orbit of the vector under circulant shifts 
and decimations~\cite{Fletcher2001}.  

\begin{comment}
 Thus, the search for an LP indexed by $\mathbb{Z}_{\ell}$ is simplified by searching only across decimation class representatives~\cite{Fletcher2001}. 
Similar to the case of LPs, Djokovic et al.~\cite{Dokovic2015}  reduced an intricate search for periodic Golay pairs among all vectors to that among {\em charm bracelets} (decimation classes). 
Hence, there is interest in the unique generation of decimation classes with fixed density.  Fletcher et al.~\cite{Fletcher2001} exhaustively generated all vectors indexed by $\mathbb{Z}_{\ell}$ for odd 
lengths $\ell\leq 47$ with density $(\ell+1)/2$.  The numbers of corresponding decimation classes were determined as a result of this search. The list of the numbers of decimation classes with density $(\ell+1)/2$ has not been expanded upon since Fletcher et al.~\cite{Fletcher2001} due to the prohibitive computational complexity of exhaustive generation.  
We provide a method for determining the number of decimation classes of density $\delta$ vectors $\vv$ that are indexed by a finite abelian group $G$ of  order $|G|$ and exponent $\ellstar$
 such that $\gcd(\delta, \ellstar)=1$.  
\end{comment}
 
For a finite abelian group $G$ with exponent $\ellstar$, let $\cI$ be the collection of all {\blue multisets with elements from $G$ with cardinality $\delta$. Then $G$ acts on $\cI$ by $I\rightarrow I+g$ for each
$I \in \cI$ and $g \in G$, where ${I+g =\{(i+g)^{\m_I(i)} \mid i \in I\}}$. } The orbit of $I$ under this action is called the {\em necklace} of $I$.
Similarly, $G\rtimes\mathbb{Z}_{\ellstar}^{\times}$
acts on $\cI$ by $I\rightarrow hI+g$ for each
$I \in \cI$ and $(g,h) \in G\rtimes\mathbb{Z}_{\ellstar}^{\times}$, where 
 ${hI=\{(hi)^{\m_I(i)} \mid i \in I\}}$ and  
$\rtimes$ is the semidirect product~\cite[p.~167]{RotmanBook}.
The orbit of $I$ under the action of $G\rtimes\mathbb{Z}_{\ellstar}^{\times} $ is called the {\em decimation class} of $I$. 
The {\em bracelet} of $I$ is defined to be the orbit of $I$ under the action 
of $G \rtimes\{-1,1\} $.

There is a one-to-one correspondence  
between decimation classes of nonnegative integer vectors indexed by $G$ with density $\delta$ and the decimation classes of elements in $\cI$ given by 
\[
{\blue \vv\rightarrow I=\{i^{v_i} \mid i \in G\} }
\]
for each $\vv \in \{\Z_{\ge 0}\}^G$ with $\sum_{i \in G}v_i=\delta$. Hence, it suffices to count the distinct decimation classes in $\cI$
to count decimation classes of nonnegative integer vectors indexed by $G$ with 
density $\delta$. We assume this one-to-one correspondence throughout this work and do not differentiate between decimation classes of density $\delta$ nonnegative integer vectors and those of the elements in $\cI$.  

\begin{comment}
We make the same assumption for necklaces and bracelets.
Titsworth~\cite{Titsworth1964} derived a formula for the number of decimation classes in $$\bigcup_{1 \leq \delta\leq |G| \atop \gcd(\delta,|G|)=1} \mathcal{I}_{\delta,G}$$ for cyclic $G$. One of  our goals in this paper is to develop a method that calculates the number of decimation classes in $\cI$ for a finite abelian group $G$ and fixed $\delta$.
\end{comment}

The {\blue multiset $I$} with elements from $G$ is said to be \textit{non-periodic} if there exists no $g \in G$, $g\neq 0$, such that $I+g=I$. To avoid trivialities, we always assume that $I$ is nonempty, and we often assume $|I|$ is finite.  The following provides a sufficient condition for $I$ to be non-periodic. 
{\blue The proof is identical to that found in \cite{decimation}, which used that $I$ had a finite number of elements.}

\begin{lma}\label{lem:necklace}
Let $G$ be a finite abelian group with exponent $\ellstar$.  Let {\blue $I$ be a multiset} with elements from $G$  such that $|I|=\delta$ and $\gcd(\delta,\ellstar)=1$. Then $I$ is non-periodic and there are $|G|$ distinct multisets with cardinality $\delta$ in each necklace in $\cI$.
\end{lma}

{\blue The following several results are analogous to those found in \cite{decimation}.  They are restatements in the context of multisets.  Their proofs, now for multisets (rather than just sets), are identical to those of Lemmas~2-5 and Theorem~1 found in \cite{decimation}, since they depend on $I$ being finite or are general group results.  Lemma~\ref{lem:unique} is a consequence of $I$ being non-periodic. The proof of Lemma~\ref{lem:|D|} depends on $I$ being finite.
The proofs of Lemmas~\ref{lma:stab2} and~\ref{lem:frob} and Theorem~\ref{thm:gcd} use group properties.}

\begin{lma}\label{lem:unique}
Let $I$ be a non-periodic multiset with elements from $G$, where $G$ is a finite abelian group with exponent $\ellstar$. Let $t \in \mathbb{Z}_{\ellstar}^\times$ be a multiplier of $I$ 
with $tI=I+g$ for some $g \in G$. 
Then $g$ is unique.  
\end{lma}

\begin{comment}
\begin{proof}
Let $tI=I+g_1=I+g_2$. Then $I+g_1-g_2=I$, and $I$ is non-periodic implies $g_1-g_2=0$, i.e., $g_1=g_2.$ 
\end{proof}

It is well known that the set of multipliers of a given {\em set} forms a subgroup of~$\mathbb{Z}_{\ellstar}^\times$, and the same holds here that the set of multipliers of~$I$ forms a subgroup of~$\mathbb{Z}_{\ellstar}^\times$.  We will denote this subgroup by~$H$.
Let $I$ be a multiset with elements from $G$ and $U_I=U_I(G)$ be the orbit of $I$ under the action of $G$; that is, let $U_I=U_I(G)$ denote the set of translates of~$I$ by elements in~$G$. For each $n \in \mathbb{Z}_{\geq 1}$, let $\phi(n)$ be the number of positive integers up to $n$ that are relatively prime to $n$. 
\end{comment}

\begin{lma}\label{lma:stab2}
{\blue Let $\mathcal{I} = \mathcal{I}_G = \cup_{\delta\ge 0} \mathcal{I}_{\delta,G}$ denote the set of all multisets with elements from $G$.  The group $\mathbb{Z}^\times_{\ellstar}$ acts on elements in
$\{U_I \mid I\in\mathcal{I} \}$ by multiplication,} where  $jU_I=U_{jI}$ 
for each $j\in\mathbb{Z}_{\ellstar}^\times$.  For a fixed $I \in\mathcal{I}$, let $H$ be the multiplier group of $I$ and $\Orb(U_I)$ be the orbit of $U_I$ under the action of $\mathbb{Z}_{\ellstar}^\times$. Then $H=\Stab(U_I)$ and $|\Orb(U_I)|=\phi(\ellstar)/|\Stab(U_I)|$.
\end{lma}

\begin{comment}
\begin{proof}
Let $I+\alpha\in U_I$.  Then $j(I+\alpha) = jI+j\alpha \in U_{jI}$.  Thus, $jU_I \subseteq U_{jI}$.  Next, let $jI+\alpha \in U_{jI}$. Since 
$j\in \mathbb{Z}_{\ellstar}^\times$, observe $jI+\alpha = j(I+j^{\phi(\ellstar)-1}\alpha) \in jU_I$.  Then $U_{jI}\subseteq jU_I$.  This proves $\mathbb{Z}_{\ellstar}^\times$ acts on the elements in $\{U_I \mid I\in\mathcal{I} \}$ by $jU_I = U_{jI}$. 

Let $H$ denote the multiplier group of $I$, and let $t\in H$.  By definition of $t\in H$, $tI=I+g$ for some $g \in G$.  Then $tI+\alpha=
I+(\alpha+g) \in U_I$ for any $\alpha\in G$.  Thus, $U_{tI}\subseteq U_I$. 
Similarly, $I+\alpha = tI+(\alpha-g)\in U_{tI}$ for any $\alpha\in G$ implying $U_{I}\subseteq U_{tI}$.   Thus, $tU_{I}=U_{tI}=U_I$ for all $t \in H$ implying that $H\subseteq \Stab(U_I)$. Now, assume $t \in \Stab(U_I)$, i.e., $U_{tI}=tU_{I}=U_I$. Then, 
$tI+\beta=I$ for some $\beta \in G$ implying $tI=I-\beta$. 
Hence,  $t \in H$ and $\Stab(U_I)\subseteq H$.  Thus, $\Stab(U_I)= H$.  Now, by the orbit stabilizer theorem, $|\Orb(U_I)| = \phi(\ellstar)/|\Stab(U_I)|$.
\end{proof}

The following lemma describes the $g$ in Definition~\ref{dfn:mul_set}  whenever ${\gcd(|I|,\ellstar)=1}$.  
\end{comment}

\begin{lma}\label{lem:|D|}
Let $I$ be a {\blue finite multiset} with elements from $G$, where $G$ is a finite abelian group with exponent $\ellstar$.  Let $t \in \mathbb{Z}_{\ellstar}^{\times}$ be such that $tI=I+g$ for some  $g \in G$ and  $\sigma=\sum_{i \in I}i$.  Then $|I| g=(t-1)\sigma$.  If $\gcd(|I|,\ellstar)=1$, then $g=|I|^{\phi(\ellstar)-1}(t-1)\sigma$. 
%where the inverse $|I|^{-1}$ is in~$\mathbb{Z}_{\ellstar}^\times$.
\end{lma}

\begin{comment}
\begin{proof}
Since $tI=I+g$, it follows that $t\sigma=\sigma+|I|g$ after separately taking sums of all elements in each multiset.  Then $(t-1)\sigma=|I|g$.  If $\gcd(|I|,\ellstar)=1$, then $|I| \in \mathbb{Z}_{\ellstar}$ and $|I|^{\phi(\ellstar)}=1$. % in~$\mathbb{Z}_{\ellstar}^\times$. 
 Therefore, $g=|I|^{\phi(\ellstar)-1}(t-1)\sigma$.
\end{proof}

The next result {\blue was proven in (cite decimation paper)} and describes the solutions of a finite abelian group equation.  
\end{comment}

\begin{lma}\label{lem:frob}
Let $G$ be a finite abelian group with $\Phi(G)=\ZZZ$, where $\Phi$ is an isomorphism. Let 
$\ellstar$ be the exponent of $G$, $a\in G$, and $m\in\mathbb{Z}$.  
Then the number of solutions to $mx=a$ for $x\in G$ is either~$0$ or $\prod_{1\le i\le r} \gcd(m,\ell_i)$. % where {$\Phi(G)=\ZZZ$} and $\Phi$ is an isomorphism.
 Moreover, a solution exists if and only if 
 $\Phi(a)_k \in  
\gcd(m,\ell_k)\mathbb{Z}_{\ell_k}$ for each $1\le k\le r$.
\end{lma}

\begin{comment}
Let $I$ be a multiset with elements from $G$.  Let $H$ be the  multiplier group of $I$.  The next two theorems provide necessary and sufficient conditions for the existence of a translate $I+z$  fixed by an element of~$H$ and by every element in~$H$,  along with the number of distinct such translates.
\end{comment}

\begin{thm}\label{thm:gcd}
Let $G$ be a finite abelian group with $\Phi(G)=\ZZZ$, where $\Phi$ is an isomorphism. Let $I$ be a multiset with elements from $G$ that is non-periodic and $t$ be a multiplier of~$I$ such that $tI=I+g$.   Then, there exists $z\in G$ such that 
$t(I+z) = I+z$ if and only~if {$\Phi(g)_k \in \gcd(t-1,\ell_k) \mathbb{Z}_{\ell_k}$} for each $1 \le k \le r$.   Moreover, if such a $z$ exists, there are {$\prod_{1\le i\le r} \gcd(t-1,\ell_i)$ such $z$'s.}
\end{thm}

\begin{comment}
\begin{proof} 
Since $I$ is non-periodic, observe $t(I+z) = I+z$ for some $z\in G$ if~and only~if $-g = (t-1)z$ for some $z\in G$.  By Lemma~\ref{lem:frob}, the latter has a solution if and only~if $$\Phi(g)_k \in \gcd(t-1,\ell_k) \mathbb{Z}_{\ell_k}$$ for each $1 \le k \le r$. 
 Moreover, by 
Lemma~\ref{lem:frob}, 
there are $\prod_{1\le i\le r} \gcd(t-1,\ell_i)$ such $z$'s 
if one such $z$ exists.
\end{proof}
\end{comment}

{\blue The next theorem also has an identical proof with the exception that
there is a one-to-one correspondence between the translates of $\Phi(I)$ fixed by  $t$ and the translates of  $I$ fixed by $t$. (It is not needed and it is not used that $\Phi(I)\subseteq \ZZZ$.)
}

\begin{thm}\label{thm:OneFixAll-dan}
Let $G$ be a finite abelian group with exponent $\ellstar$ and $\Phi(G)=\ZZZ$, where $\Phi$ is an isomorphism. Let $I$ be a multiset with elements from $G$ that is non-periodic with multiplier group $H$. Let 
$K=\langle t_1,\ldots, t_m\rangle\leq H$, $t_iI=I+g_i$ for $1\le i\le m$, and  {$C=\gcd(t_1-1,\ldots,t_m-1,\ellstar)$.} Then the following hold.
\begin{enumerate}
%\item  For each $1\le i\le m$ there exists some $z_i\in G$ such that $I+z_i$ is fixed by $t_i$ if and only if for each $1\le i\le m$ and $1\le k\le r$, $( \Phi(g_i) )_k \in \gcd(t_i-1,\ell_k) \mathbb{Z}_{\ell_k}$. \label{cond:exits}

 \item If $C=1$, then there exists a translate $I+z$ fixed by all multipliers $t\in K$.\label{cond:unique}
 \item For each $1 \le j \le r$, let $1 \le i_j \le \gcd(C,\ell_j)$ and $\Phi( h_j' )$ be defined such that
  $$\langle \Phi(h'_j)\rangle=
 \langle0\rangle\times\cdots\times\mathbb{Z}_{\ell_j}\times\cdots\times\langle0\rangle.$$  If a $z'_0\in G$ exists such that $I+z_0'$ is fixed by $K$, then
%  all the translates 
  $I+z'$ 
  is fixed by $K$ 
  %satisfy  
  for
$$z'=z'_0+h'_1i_1\frac{\ell_1}{\gcd(C,\ell_1)}+\cdots+h'_ri_r\frac{\ell_r}{\gcd(C,\ell_r)}. $$
%for each $ 1 \leq i_j \leq \gcd(C,\ell_j).$ 
  Moreover, these are the precisely  $\prod_{1\le i\le r}\gcd(C,\ell_i)$ distinct such $z'\in G$. \label{cond:count} 
 \end{enumerate}
\end{thm}

{\blue The following results are the multiset versions of those found in~\cite{decimation}.  The proofs are identical for finite multisets with the same properties}. 

\begin{lma}\label{lem:|D|dan}
Let $G$ be a finite abelian group with exponent $\ellstar$ and $\Phi(G) = \ZZZ$, where $\Phi$ is an isomorphism. 
Let $I$ be a finite, non-periodic multiset with elements from~$G$ and $\sigma=\sum_{i\in I} i$.  
Let $t\in \mathbb{Z}_\ellstar^\times$ be such that $tI = I + g_0$ for some $g_0\in G$. 
For each $1 \le k \le r$, let $d_k = \gcd(|I|,\ell_k)$ and $e_k$ be the $k$'th column of the $r\times r$ identity matrix.  
Then, $(t-1)\Phi(\sigma)_k / d_k \in \mathbb{Z}$ for each $1\le k\le r$. Moreover, $g_0$ is unique and one of the $\prod_{1\le i\le r} \gcd(|I|,\ell_i)$ solutions to $|I| g = (t-1)\sigma$ for $g\in G$, namely \[g = \sum_{1\le k\le r}  \left(\left( \frac{| I | }{d_k} \right)^{\phi\left(\frac{\ell_k}{d_k}\right)-1} \frac{(t-1) \Phi(\sigma)_k}{d_k} + \frac{\ell_k}{d_k} \, j_k \right) \Phi^{-1}(e_k)\]where each choice of $j_k$ satisfies $0\le j_k < d_k$ for each $1\le k\le r$. 
\end{lma} 

Based on Lemmas~\ref{lem:unique} and~\ref{lem:|D|dan}, we have the following definition.
\begin{dfn}\label{dfn:Jt}
Let $G$ be a finite abelian group with exponent $\ellstar$ and $\Phi(G) = \ZZZ$, where $\Phi$ is an isomorphism. 
Let $I$ be a finite, non-periodic multiset with elements from~$G$ and $\sigma=\sum_{i\in I} i$. 
Let $t\in \mathbb{Z}_\ellstar^\times$ be such that $tI = I + g_0$ for some $g_0\in G$. 
For each $1 \le k \le r$, let $d_k = \gcd(|I|,\ell_k)$ and $e_k$ be the $k$'th column of the $r\times r$ identity matrix.  
Then, define $j_1(t),\ldots,j_r(t)$ to be the necessarily unique integers (by 
Lemmas~\ref{lem:unique} and~\ref{lem:|D|dan}) such that $0\leq j_k<d_k$ for each $1 \leq k\leq r$ and
 $$g_0=\sum_{1\le k\le r}\left(\left(\frac{|I|}{d_k}\right)^{\phi\left(\frac{\ell_k}{d_k}\right)-1}\frac{(t-1)\Phi(\sigma)_k}{d_k} + \frac{\ell_k}{d_k}j_k(t)\right)\Phi^{-1}(e_k).$$ 
\end{dfn}

%Next, we provide a necessary and sufficient condition for a multiplier $t$ of $I$ to fix at least one translate $I+z$ for some $z\in G$ when $G$ is a finite abelian  group with exponent $\ellstar$.   

\begin{thm}\label{cyclicthm}
Let $G$ be a finite abelian group with exponent $\ellstar$ and $\Phi(G) = \ZZZ$, where $\Phi$ is an isomorphism. 
Let $I$ be a finite, non-periodic multiset with elements from~$G$.  Let $t\in \mathbb{Z}_\ellstar^\times$ be a multiplier of $I$ and $\sigma=\sum_{i\in I} i$.  
For each $1 \le k \le r$, let $d_k = \gcd(|I|,\ell_k)$ and $e_k$ be the $k$'th column of the $r\times r$ identity matrix.  
Let $j_1(t),\ldots,j_r(t)$ be as in Definition~\ref{dfn:Jt}. Then,  $z\in G$ is a solution to $t(I+z)=I+z$ if and only~if  $z\in G$ is a solution to 
 \begin{equation}\label{tgt}
(t-1)z = -\sum_{1\le k\le r}\left(\left(\frac{|I|}{d_k}\right)^{\phi\left(\frac{\ell_k}{d_k}\right)-1}\frac{(t-1)\Phi(\sigma)_k}{d_k} + \frac{\ell_k}{d_k}{j_k}(t)\right)\Phi^{-1}(e_k).
 \end{equation}
\end{thm}

\begin{comment}
\begin{proof}
 By Lemma~\ref{lem:|D|dan} and Definition~\ref{dfn:Jt}, if $t$ is a multiplier of $I$, then  $$tI=I + 
 \sum_{1\le k\le r}\left(\left(\frac{|I|}{d_k}\right)^{\phi\left(\frac{\ell_k}{d_k}\right)-1}\frac{(t-1)\Phi(\sigma)_k}{d_k} + 
 \frac{\ell_k}{d_k}{j_k}(t)\right)\Phi^{-1}(e_k).$$
Since $I$ is a non-periodic multiset, we deduce that $t(I+z) = I + z$ if and only if  $z\in G$ is a solution to equation~(\ref{tgt}).
\end{proof}

In an attempt to simplify the condition in 
Theorem~\ref{cyclicthm} and derive practical results, we prove the following corollary and the subsequent theorem.
\end{comment}

\begin{crly}\label{fixtrans}
Let $G$ be a finite abelian group with exponent $\ellstar$ and $\Phi(G) = \ZZZ$, where $\Phi$ is an isomorphism. 
Let $I$ be a finite, non-periodic multiset with elements from~$G$.  Let $t\in \mathbb{Z}_\ellstar^\times$ be a multiplier of $I$ and $\sigma=\sum_{i\in I} i$.  
For each $1 \le k \le r$, let $d_k = \gcd(|I|,\ell_k)$ and $e_k$ be the $k$'th column of the $r\times r$ identity matrix.  
Let $j_1(t),\ldots,j_r(t)$ be as in Definition~\ref{dfn:Jt}. 
Then, $t$ fixes at least one translate $I + z$ for some $z\in G$ if and only if
\begin{equation}\label{div-cond}
d_k\cdot \gcd(t-1,\ell_k) {\rm~divides~} \left( \frac{| I |}{d_k} \right)^{\phi\left(\frac{\ell_k}{d_k}\right)-1} (t-1) \Phi(\sigma)_k + \ell_k \, j_k(t)
\end{equation}
for each $1\le k\le r$.
\end{crly}

\begin{comment}
\begin{proof}
To ease some of the notation, let 
\[
b_k = \left( \frac{| I |}{d_k} \right)^{\phi\left(\frac{\ell_k}{d_k}\right)-1} \text{~for each $1\le k\le r$.}
\]
Since $\Phi$ is an isomorphism, a solution $z\in G$ to~(\ref{tgt}) %(\ref{10inThm3}) %CHANGE THIS
exists if and only if there exists $\Phi(z)\in \ZZZ$ such that
\[
(t-1)\Phi(z)_k \equiv - b_k \frac{(t-1) \Phi(\sigma)_k}{d_k} - \frac{\ell_k}{d_k} \, j_k(t) \quad \text{(mod } \ell_k) \text{~for each $1\le k\le r$.}
\]
This system of congruences has a solution $\Phi(z)$ if and only if the integer divisibility 
\[
\gcd(t-1,\ell_k) \text{~divides~} b_k \frac{(t-1) \Phi(\sigma)_k}{d_k} + \frac{\ell_k}{d_k} \, j_k(t)
\]
holds for each $1\le k\le r$.  Multiply  by $d_k$ for $1\le k\le r$ to obtain the equivalent condition in~(\ref{div-cond}).
\end{proof}
\end{comment}

\begin{thm}\label{thm:Dplusalpha}
Let $G\cong\ZZZ$ have exponent $\ellstar$.  
Let $I$ be a finite, non-periodic multiset with elements from~$G$ such that $\gcd(|I|,\ellstar)=1$, and let $t\in \mathbb{Z}_\ellstar^\times$ be a multiplier of~$I$.   
Then, there are  exactly  $\prod_{1\le i\le r} \gcd(t-1,\ell_i)$ distinct $z\in G$ such that $t(I+z)=I+z$. Moreover, 
$z = -|I|^{\phi(\ellstar)-1} (\sum_{i \in I}i)$ is a solution that does not depend on $t$. Hence, there is a $z\in G$ such that  $I+z$ is fixed by the multiplier group $H$ of $I$.
\end{thm}

\begin{comment}
\begin{proof}
Since $\gcd(|I|,\ellstar)=1$, $I$ is non-periodic by Lemma~\ref{lem:necklace}.
By Theorem~\ref{cyclicthm}, it suffices to find the solutions to equation (\ref{tgt}).  
Let $\sigma=\sum_{i \in I}i$. Then, equation~(\ref{tgt}) becomes 
\begin{equation}\label{tgtd1}
(t-1)z = -|I|^{\phi(\ellstar)-1}(t-1)\sigma.
\end{equation} 
Observe that equation~(\ref{tgtd1}) has $$z_0=-|I|^{\phi(\ellstar)-1}\sigma \in G$$ as a solution that does not depend on $t$. 
Then, there is a $z\in G$ such that $I+z$ is fixed by the multiplier group $H$ of $I$. Now, the
solutions $z\in G$ such that $I+z$ is fixed by $\langle t \rangle$ are given in Theorem~\ref{thm:OneFixAll-dan} part~\ref{cond:count}. 
Since $\gcd(\gcd(t-1,\ellstar),\ell_i) = \gcd(t-1,\ell_i)$ for each $1\le i\le r$, there are exactly $\prod_{1\le i\le r} \gcd(t-1,\ell_i)$ distinct such solutions in this case. Thus, there are at least this many solutions $z\in G$ such that $I+z$ is fixed by just $t$. By Lemma~\ref{lem:frob}, %\ref{lemma5},
there are no more solutions to~(\ref{tgtd1}). %\ref{eqn12}).
Hence, the solutions given in Theorem~\ref{thm:OneFixAll-dan} part~\ref{cond:count}
provides all the solutions to~(\ref{tgtd1}). 
\end{proof}
\end{comment}

\section{A sufficient condition for fixed translates from the adjacency matrix} \label{sec:adjacencymatrix}
A multiplier $t$ of $I$ is called {\em translate fixing} if there exists at least one $z \in G$ such that $t(I+z)=I+z$. 
 First, we introduce the concept of the adjacency matrix of a multiset $I$, where $I$ has elements from a finite abelian group $G$. This is then used to derive a sufficient condition for each multiplier of $I$ to be translate fixing. 
 
For a finite multiset $I$ with elements from $G$, the \textit{adjacency matrix} of $I$, denoted by $\T_I$, is defined by 
\begin{equation}\label{eqn:IncidenceMatrix}
\T_I(i,j)=\m_{I+g_i}(g_j)
\end{equation}
for $0\le i,j\le |G|-1$; that is, $\T_I(i,j)$ is defined to be the multiplicity of $g_j$ in $I+g_i$.

Let $\mP_g$ and $\mQ_t$ be permutation matrices such that
\begin{equation}\label{PQeqns}
\mP_g\T_I = \T_{I+g}~\quad\text{and}\quad~~\mQ_t^\top\T_I\mQ_t = \T_{tI}.
\end{equation}
These permutation matrices have the following properties, which will be used repeatedly.
\begin{itemize}
	\item $\mP_g\mP_h=\mP_{h+g}=\mP_h\mP_g$
	\item $\mQ_s^\top\mQ_t^\top\T_I\mQ_t\mQ_s = \mQ_{st}^\top\T_I\mQ_{st} = \mQ_t^\top\mQ_s^\top\T_I\mQ_s\mQ_t$
	\item $\mQ_t^\top\mP_g\mQ_t=\mP_{tg}$
	\item $\mP_{g}^\top=\mP_{g}^{-1}=\mP_{-g}; \quad \mQ_{t}^\top=\mQ_{t}^{-1}=\mQ_{t^{-1}}$
\end{itemize}

{\blue The following lemmas demonstrate how the permutation matrices $\mP_g$ and $\mQ_t$ act on $\T_I$.  The proof of the next result is the same as that found in~\cite{decimation}.}

\begin{lma}\label{lma:stab1}
Let $G$ be a finite abelian group with exponent $\ellstar$.  Let $I$ be a finite multiset with elements from $G$.  If $t\in \mathbb{Z}^\times_{\ellstar}$ is a multiplier of $I$, then $t$ is a multiplier of every translate of~$I$.
\end{lma}

{\blue The proof of the next result is the same as that found in~\cite{decimation} with the exception that the $r$th row of $\T_I$ is identified by a translate of $I$, and each multiset identifying a row of $\T_I\Q_t$ is obtained by multiplying a multiset identifying a row of $\T_I$ by $t$ in $\mathbb{Z}^{\times}_{\ellstar}$. The remaining results in this seciton have identical proofs.}

\begin{thm}\label{thm:fixtrans}
Let $G$ be a finite abelian group with exponent $\ellstar$, $I$ be a finite multiset with elements from $G$, and $\T_I$ be the adjacency matrix from~(\ref{eqn:IncidenceMatrix}).  If $\T_I$ is invertible, then each multiplier of $I$ is translate fixing.  
\end{thm}

\begin{thm}\label{thm:TIcyclic}
Let $G$ be a cyclic group with exponent $\ellstar$, $I$ be a finite multiset with elements from $G$, and $\T_I$ be the adjacency matrix from~(\ref{eqn:IncidenceMatrix}). 
Suppose $[c_{g_0},c_{g_1},\ldots,c_{g_{\ell-1}}]$ is the first column of the adjacency matrix $\T_I$.  Then, $\T_I$ is invertible if and only if $\gcd(\sum_{0\le j\le\ell-1}c_{g_j}x^j,x^\ell-1)$ in $\mathbb{Q}[X]$ is constant.
\end{thm}

\begin{comment}
\begin{proof}
 Let $\T_I$ denote the adjacency matrix of~$I$. 
Then $\T_I$ is a circulant matrix, say
\[
\begin{bmatrix}
    c_{g_0} & c_{g_{\ell-1}} & c_{g_{\ell-2}} & \dots  & c_{g_1} \\
    c_{g_1} & c_{g_0} & c_{g_{\ell-1}} & \dots  & c_{g_2} \\
    c_{g_2} & c_{g_1} & c_{g_0} & \dots  & c_{g_3} \\
    \vdots & \vdots & \vdots & \ddots & \vdots \\
    c_{g_{\ell-1}} & c_{g_{\ell-2}} & c_{g_{\ell-3}} & \dots  & c_{g_0}
\end{bmatrix}.
\] 
Let $P(x)=c_{g_0}+c_{g_1}x+c_{g_2}x^2+\cdots+c_{g_{\ell-1}}x^{\ell-1}$. 
 It is known that the rank of a circulant $\ell\times\ell$ matrix equals $\ell-d$ where $d$ equals the degree of $\gcd\left(P(x),x^\ell-1\right)$ in $\mathbb{Q}[X]$~\cite{Ingleton1956}.  Since the circulant matrix $\T_I$ has rank~$\ell$ if and only if $\T_I$ is invertible, the result follows. 
\end{proof}

Now, we get the following corollary by Theorems~\ref{thm:fixtrans} and~\ref{thm:TIcyclic}.
\end{comment}

\begin{crly}
Let $G$ be a cyclic group with exponent $\ellstar$, $I$ be a finite multiset with elements from $G$, and  $[c_{g_0},c_{g_1},\ldots,c_{g_{\ell-1}}]$ be the first column of the adjacency matrix $\T_I$.  
If $\gcd(\sum_{0\le j\le\ell-1}c_jx^j,x^\ell-1)$ in $\mathbb{Q}[X]$ is constant, then each multiplier of $I$ is translate fixing.
\end{crly}

\section{Counting necklaces and bracelets}\label{sec:CountingNecks} 
Let $N(\ell,\ellstar,\delta)$ be the number of necklaces in $\cI$ for a finite abelian group $G$ of order $\ell$ and exponent $\ellstar$.
Then, by Lemma~\ref{lem:necklace}, 
$N(\ell,\ellstar, \delta)=\multiset{\ell}{\delta}/\ell$ whenever $\gcd(\delta,\ellstar)=1$. Thus, if  $\gcd(\ellstar,\delta)=1$, then each necklace  is guaranteed to contain exactly $\ell$ vectors.
Since each necklace contains $|G|$ vectors and each bracelet contains at most two necklaces, each bracelet contains at most $2|G|$ vectors.  A vector $\vv$ is called \textit{symmetric} if there exists some $j_0\in G$ such that $${v_{j_0+k\,}=v_{j_0-k\, } \ \forall \ k\in G}.$$ Such a $j_0$ is called the {\em index of symmetry} of $\vv$. If an index is not an index of symmetry, then it is called an {\em index of non-symmetry}. It is possible for a vector to have more than one index of symmetry. Observe that indices of symmetry of a vector $\vv$ are the only 
indices $i$ such that $v_i=v_{i'}$ does not necessarily hold for some $i' \in G$ with $i' \neq i$.

 If a vector within a necklace is symmetric, then all other vectors in the necklace are also symmetric.  Such a necklace is defined to be {\em symmetric}.  A bracelet contains a single necklace if and only if that necklace is symmetric. Hence we get the following fact.
\begin{fact}\label{fact:symnecklaces}
 The number of density $\delta$  symmetric necklaces of length $\ell$ is  the same as the number of density $\delta$ symmetric bracelets of length $\ell$.
 \end{fact} 

{\blue The following lemma determines the form and the number of 
indices of symmetry for a vector $\vv$ indexed by a finite abelian group $G$. 
The proof from \cite{decimation} holds for such vectors.}

\begin{lma}\label{lem:indsym}
Let $G\cong \ZZZ$ be a finite abelian group of order
$\ell=\ell_1 \cdots\ell_r$ with exponent $\ellstar=\lcm(\ell_1,\ldots,\ell_r)$.
Let $\vv$ be a vector indexed by $G$ and symmetric. Then the number of indices of symmetry for $\vv$ is  $\prod_{i=1}^r\gcd(2,\ell_i)$. Moreover, if $j_1$ is an index of symmetry for $\vv$, then every other index of symmetry $j_2$ has the form
$j_2=j_1+\Delta$, where $\Delta$ is a solution to the equation 
\begin{equation} \label{eqn:Delta}
 2\Delta=0, \quad \Delta \in G.
\end{equation}  
\end{lma}

\begin{comment}
\begin{proof}
Let $j_1,j_2$ be   indices of symmetry of $\vv$.
 Let $j_1+\Delta=j_2$. Since the indices of symmetry of a vector $\vv$ are the only 
indices $i$ such that $v_i=v_{i'}$ 
does not necessarily hold for some $i' \in G$ with $i' \neq i$, we must have 
 $$j_1+\Delta=j_1-\Delta=j_2$$  implying
%\begin{equation*} \label{eqn:Delta}
$ 2\Delta=0,$ $ \Delta \in G.$
%\end{equation*}  
Then by Lemma~\ref{lem:frob}, the number of distinct solutions to equation~(\ref{eqn:Delta})
is $\prod_{i=1}^r\gcd(2,\ell_i)$ as $\Delta=0$ is always a solution.
Hence, there are $\prod_{i=1}^r\gcd(2,\ell_i)$ distinct possibilities for
$j_2$ including the case $j_1=j_2$. Moreover, $j_2=j_1+\Delta$, where $\Delta$ is a solution to equation~(\ref{eqn:Delta}).
\end{proof}
\end{comment}

{\blue The following theorem determines the number of binary  symmetric necklaces of length $\ell$ and density $\delta$ when $\gcd(\ellstar,\delta)=1$.  It generalizes the previous result found in~\cite{decimation}. }

\begin{thm}\label{thm:numSymNecks}
Let $G\cong \ZZZ$, $|G|=\ell=\prod_{i=1}^r\ell_i$, $\ellstar=\lcm(\ell_1,\ldots,\ell_r)$ be the exponent of $G$, and $\theta=\prod_{i=1}^r\gcd(2,\ell_i)$.
 Then the number of density $\delta$ symmetric necklaces in $\{\Z_{\ge 0}\}^G$ 
 %such that $\gcd(\delta,\ellstar)=1$ 
 is
\begin{equation}\label{eqn:NumSymNecks}
\sum_{\delta_1+2\delta_2=\delta} \multiset{\frac{\ell-\theta}{2}}{\delta_2} \multiset{\theta}{\delta_1}
\end{equation} 
where the sum is over all $\delta_1,\delta_2\in\Z_{\ge 0}$ such that $\delta_1+2\delta_2=\delta$.
\end{thm}
 
\begin{proof}
First, by Fact~\ref{fact:symnecklaces}, counting symmetric necklaces of length $\ell$ 
is the same as counting symmetric bracelets of length $\ell$.
Let $\vv\in \{\Z_{\ge 0}\}^G$ and have density $\delta$. By Lemma~\ref{lem:indsym}, the number of indices of symmetry for $\vv$ is $\theta=\prod_{i=1}^r\gcd(2,\ell_i)$. 
Let $j_1,\ldots,j_{\theta}$ be all the indices of symmetry of $\vv$.  
Let $I_{non}=G \setminus \{j_1,\ldots j_{\theta}\}$. 
Let $I^{1/2}_{non}\subseteq I_{non}$ be such that $I^{1/2}_{non} \cap -I^{1/2}_{non}=\emptyset$ and $|I^{1/2}_{non}|=(\ell-\theta)/2$. Then the entries of $\vv$ on $I^{1/2}_{non}$ determine the entries of $\vv$ on $-I^{1/2}_{non}$ and $I^{1/2}_{non} \cup -I^{1/2}_{non}=I_{non}$. 
Thus, $\vv$ is completely determined by the $\theta$ indices of symmetry and the indices in $I^{1/2}_{non}$. 

For a vector $\vv$, let $\delta_1$ and $\delta_2$ be the sum of multiplicities of indices of symmetry and non-symmetry in $I^{1/2}_{non}$, respectively.   More precisely, as $\delta = \sum_{i\in G}v_i$,
\begin{equation}\label{deltas}
\delta_1 = \sum_{i=1}^{\theta}v_{j_i} \quad\text{and}\quad \delta_2 = \sum_{i\in I^{1/2}_{non}} v_i. 
\end{equation}
Then the density of $\vv$ is $\delta=\delta_1+2\delta_2$.  
By the comment including (\ref{2.1.1}) and the fact that $|I^{1/2}_{non}|=(\ell-\theta)/2$, there are $\multiset{\theta}{\delta_1} \multiset{\frac{\ell-\theta}{2}}{\delta_2}$
choices of vectors $\vv$ with density $\delta$ such that the equations in (\ref{deltas}) hold.
Hence, as $\delta_1$ and $\delta_2$ vary, the number of density $\delta$ symmetric necklaces in $\{\Z_{\ge 0}\}^G$ %such that $\gcd(\delta,\ellstar)=1$ 
given in equation~(\ref{eqn:NumSymNecks}) holds.
\end{proof}

The next corollary follows immediately from Theorem~\ref{thm:numSymNecks}. 
\begin{crly}
Let $G\cong \mathbb{Z}_{\ell}$. Then the number of  density $\delta$ symmetric necklaces in $\{\Z_{\ge 0}\}^G$ 
%such that $\gcd(\delta,\ellstar)=1$ 
is
\begin{equation*}\label{eqn:NumSymNeckscrly}
\eta=\multiset
{\left\lfloor\frac{\ell-1}{2}\right\rfloor} {\left\lfloor\frac{\delta}{2}\right\rfloor}.
\end{equation*}
\end{crly}

\begin{comment}
Hence, the number of length $\ell$ bracelets with density $\delta$ 
% when $\gcd(\ellstar,\delta)=1$  is
\begin{equation*}
\gamma=\frac{
%\left(\begin{matrix}
{\ell \choose
%\\ 
\delta}
%\end{matrix}  \right)
 }{2\ell }+\frac{\eta}{2},
\end{equation*}
where $\eta$ denotes the number of binary symmetric necklaces.
\end{comment}

\section{$H$-orbits}\label{sec:Ringcosets}
For a subgroup $K=\{t_1,\ldots, t_{|K|}\}\leq\mathbb{Z}^\times_{\ellstar}$ and $s\in G$, let $sK=\{st_1,\ldots, st_{|K|}\}$ be the \textit{$K$-orbit of~$s$}. The next theorem follows  from Theorem~\ref{thm:Dplusalpha}.
\begin{thm}\label{crly:union}
Let $G\cong\ZZZ$ have exponent $\ellstar$, $K\leq\mathbb{Z}^\times_{\ellstar}$ be a subgroup of the group of all multipliers of a finite multiset $I$ with elements from $G$, and $\sigma=\sum_{i \in I}i$.   
If $\gcd(|I|,\ellstar)=1$, then
\begin{equation}\label{eqn:union}
I-|I|^{\phi(\ellstar)-1}\sigma=\alpha_1 s_1K\cup \alpha_2 s_2K\cup \cdots \cup \alpha_r s_rK
\end{equation}
for some $s_1, s_2, \ldots, s_r \in G$ and $\alpha_1, \alpha_2, \ldots, \alpha_r \in \Z_{\ge0}$, where the union in equation~(\ref{eqn:union}) is disjoint.
\end{thm}

\begin{proof}
By Theorem~\ref{thm:Dplusalpha},
\begin{equation*}\label{eqn:unionAll}
t(I-|I|^{\phi(\ellstar)-1}\sigma)= I-|I|^{\phi(\ellstar)-1}\sigma \quad \text{for all } t \in K,
\end{equation*}
and $K$ acts on the elements in $I-|I|^{\phi(\ellstar)-1}\sigma$. Then equation~(\ref{eqn:union}) is the decomposition of $I-|I|^{\phi(\ellstar)-1}\sigma$ into disjoint union of orbits under the action of $K$, {\blue where each $\alpha_i$ is the multiplicity over $I-|I|^{\phi(\ellstar)-1}\sigma$ of the elements is $s_i K$}.
\end{proof}

The following lemma is an application of the orbit-stabilizer theorem.  Its proof is identical to that found in~\cite{decimation}, which relied only on the group action.

\begin{lma} \label{lma:SetDividesGroup}
If $K\leq\mathbb{Z}^\times_{\ellstar}$ and $s \in G$, then $|s K|$ divides 
$|K|$.
\end{lma}

\begin{comment}
\begin{proof}  The group $K$ acts on the elements in $s K$ by multiplication, where for each $t_1,t_2\in K$ and $z \in sK$, $(t_1t_2)z =t_1(t_2z )$. Under this action, for any $st_0 \in sK$, $\Orb(st_0)=sK$.  By the orbit-stabilizer theorem, $|s K|$ divides $|K|$.
\end{proof}

Next, we focus on $\cI$, where $G$ is a finite abelian group with exponent $\ellstar$ such that $\gcd(\delta,\ellstar)=1$.  
\end{comment}

{\blue The following definition relates vectors to multisets and allows the extension of the algorithm from~\cite{decimation} to work for nonnegative integer vectors.}

\begin{dfn}\label{dfn:mul_vect}
Let $G$ be a finite abelian group. An integer $t\in \mathbb{Z}_{\ellstar}^\times$ is called a \textit{multiplier of a vector} 
$\vv \in\{\Z_{\ge 0}\}^G$ if $t$ is a multiplier of the multiset 
$I_{\vv }=\{i^{v_i}\mid i \in G\}$ which contains $i\in G$ $v_i$ times. 
\end{dfn}

 It follows from Definition~\ref{dfn:mul_vect} that $t$ is a multiplier of $\vv$ if and only if $d_{t^{-1}}(\vv)\in U_{\vv }$. 
  Since the set of multipliers of vectors forms a subgroup of $\mathbb{Z}_{\ellstar}^\times$, 
  $t$ is a multiplier of $\vv$ if and only if $t^{-1}=t^{\phi(\ellstar)-1}$ is a multiplier of $\vv$.  Then $t$ is a multiplier of $\vv$ if and only if $d_{t}(\vv)\in U_{\vv}$.

Next, given a potential multiplier group
 $H \le \mathbb{Z}_{\ellstar}^\times$, we 
 determine the number of necklaces $U_{\vv}$ with multiplier group $H$. By  Lemma~\ref{lma:stab2} and Theorem~\ref{crly:union}, finding each necklace $U_{\vv}$ with multiplier group $\Stab (U_{\vv})= H$ is equivalent to finding each collection of $H$-orbits  whose combined size is $|I_{\vv}|$.  

Let $G$ be a finite abelian group, and consider any $\vv\in\{\Z_{\ge 0}\}^G$ with multiplier group $H\leq\mathbb{Z}^\times_{\ellstar}$ such that $\gcd(\delta,\ellstar)=1$ where $\delta := |I_{\vv}|$.  Since $I_{\vv}$ is a multiset with elements from $G$ and $\gcd(|I_{\vv}|,\ellstar)=1$, by Theorem~\ref{thm:Dplusalpha}, 
 there exists a translate of $I_{\vv}$ fixed by~$H$.
Thus, when searching for necklace representatives, it
suffices to search for vectors $\vv$ such that $I_{\vv}$ is
fixed by $H$ by replacing $\vv$ with $\vv'$ 
where $v_g'=\m_{I'}(g)$ and $I'=I_{\vv} - | I_{\vv} |^{\phi(\ellstar)-1}(\sum_{i\in I_{\vv}} i)$.
The finite abelian group $G\cong\ZZZ$ has a decomposition into $H$-orbits; that is, $G=\bigcup_{i=1}^{e}s_iH$ is a disjoint union of $H$-orbits.
Since the $H$-orbits are disjoint, there exists $\{f_1,f_2,\ldots, f_b\}\subseteq \{1,2,\ldots, e\}$ such that $$I_{\vv}=r_1s_{f_1}H\cup r_2s_{f_2}H\cup \cdots \cup r_b s_{f_b}H$$
for some positive integers $b$ and $r_1,r_2,\ldots,r_b\in\Z_{\ge 1}$ by Theorem~\ref{crly:union}.
In general, let 
\[
x_{s_i}=\begin{cases}
 r_i \quad \text{if $v_j=r_i$ for all $j\in s_iH$}\\
 0 \quad \text{otherwise}
 \end{cases}
\]
and $a_i=|s_iH|$ for each $1 \le i \le e$.  Then by Theorem~\ref{crly:union}, the number of solutions to the binary integer linear program for $H$ (ILP$_H$)
\begin{equation}\label{ILP}
\begin{array}{rl}
\min & 0  \\
\mbox{subject to:} &  \sum_{i=1}^ea_ix_{s_i}=\delta,~x_{s_i}\in\Z_{\ge 0}
\end{array}
\end{equation}
is equal to the number of possible $I_{\vv}$ whose multiplier group contains $H$ and satisfies $\delta=|I_{\vv}|$. Here the zero objective function is chosen to cast the problem of finding the number of possible $I_{\vv}$ whose multiplier group contains $H$ and satisfies $\delta=|I_{\vv}|$ as the problem of counting the number of solutions to a ILP.
ILP$_H$~(\ref{ILP}) is a formulation of a sum problem (SP). 
Determining whether ILP$_H$~(\ref{ILP}) is feasible is known to be NP-complete~\cite{Alfonsin1998}.  Finding all solutions of ILP$_H$~(\ref{ILP}) is NP-hard.  The set of solutions to ILP$_H$~(\ref{ILP}) potentially has multiple $j$ such that $I_{c_j(\vv)}$ satisfies 
equation~(\ref{eqn:union}) for some $\vv$.  The repetitions of such solutions are due to the  translates of $I$ fixed by $H$.  By Theorem~\ref{thm:OneFixAll-dan}, the number of such translates  is $\prod_{1\le i\le r} \gcd(C,\ell_i)$, where $C=\gcd(t_1-1,t_2-1,\ldots,t_m-1,\ellstar)$, $H=\langle t_1,t_2,\dots,t_r\rangle$, and $\ellstar$ is the exponent of~$H$.  

Only the number of solutions to ILP$_H$~(\ref{ILP}) is needed  for the purpose of counting decimation classes.  Since $|\mathbb{Z}_{\ellstar}^\times/H|=|\mathbb{Z}_{\ellstar}^\times|/|H|$ and $|s_iH|=|H|$ if 
$s_i\in \mathbb{Z}_{\ellstar}^\times$, for each multiplier group $H$, there exists at least $|\mathbb{Z}_{\ellstar}^\times|/|H|$ $H$-orbits of size $|H|$.  Furthermore, each $H$-orbit has size dividing $|H|$.  Hence, there exists significant duplicity among SP set values, i.e., among the
elements of the sequence  $\{a_i\}_{i=1}^e$. 
 The number of solutions to ILP$_H$~(\ref{ILP}) is computed more efficiently by determining only the number of solutions which are unique up to permutations of the variables of ILP$_H$~(\ref{ILP}) with the same constraint coefficients.  We call this reformulated problem the \textit{unique sum problem} (USP).

In Algorithm~\ref{alg:recursion} (Recursion), $\text{zeros}(n,1)$ is a vector of all zeros of length $n$.
For a vector~$\xx$,  let $\{\xx\}$ be the set of entries of $\xx$ without repetitions.
The vector $\C_{H}$  stores the sizes of all possible $H$-orbits of $s$ for the multiplier group $H$ and $s \in G$ sorted in ascending order. In Method~\ref{meth:USP},  
called {\em unique sum} (US),
 $\Q_H=\text{uniquesort}(\C_{H})$ is a vector whose entries are all the  elements of the set $\{\C_{H}\}$ sorted in ascending order, i.e.,
  $\Q_H$ is obtained by removing the repeat values in the vector $\C_{H}$.    
 The length of a vector $\xx$ is $\text{length}(\xx)$ and 
 $\sum_{i}[f_H(i)=q_H(j)]$ is the number of entries of $\C_H$ that are equal to $q_H(j)$.
 For any ordered solution $\PP$ to the US we must have $\{\PP\}\subseteq \{\C_H\}$ and satisfy the additional constraint, $p_{s_i}\leq p_{s_{i+1}}$ for each $p_{s_i}\in \{\PP\}$. Each $p_{s_i}\in \{\PP\}$ corresponds to some $a_jx_{s_j}$ in 
ILP$_H$~(\ref{ILP}) such that $x_{s_j}=r_j$.  Such $H$-orbit combinations are  obtained in Algorithm~\ref{alg:recursion} via recursion using the vector $\Q_H$ and  their corresponding duplicity in $\{\C_H\}$, denoted by $\R_H$. 
 At the $k$'th  stage of the recursion, the desired sum is $\mu$ and the index of $\Q_H$ being considered for addition is $k$.

\begin{algorithm}[H]
\caption{Counting all solutions to the SP by counting that of USP (Recursion)}\label{alg:recursion}
\begin{algorithmic}[1]
\Procedure{Recursion}{$\Q_H$,$\R_H$,$\mu$, $k$}
\State $\sols := \text{zeros}(r_H(k),1$);
%\For{$j:= 0$ {\bf to} $m$ {\bf step} $1$}
\For {$j_k:=0$ {\bf to} $\lfloor \frac{\mu}{q_H(k)} \rfloor$ {\bf step} $1$}
\State $\nu := \mu-j_k*q_H(k)$;
\If{$\nu=0$}
\State $sols(j_k):=\multiset{r_H(k)}{j_k}$;
\State break;
\EndIf
\If{$\nu<0$} 
\State break;
\EndIf
\State $sols(j_k) := \text{Recursion}(\Q_H,\R_H,\nu,k+1)$;
\State $sols(j_k) := sols(j_k)*\multiset{r_H(k)}{j_k}$;

\EndFor

\Return $\sum_{j_k}sols(j_k)$;
\EndProcedure
\end{algorithmic}
\end{algorithm}

\newfloat{method}{htbp}{loa}
\floatname{method}{Method}

\begin{method}[H]%[t]
\caption{Unique sum (US)\label{meth:USP}}
\begin{algorithmic}[1]
\Procedure{US}{$\C_H$, $\delta$}
\State $\Q_H := \text{uniquesort}(\C_H)$;
\State $m:=\text{length}(\Q_H)$;
\State $\R_H:=\text{zeros}(m,1)$;
\For{$j:= 0$ {\bf to} $m-1$ {\bf step} $1$}
\State $r_H(j):=\sum_{i}[f_H(i)=q_H(j)]$;

\EndFor

\Return $z:=\, $Recursion($\Q_H,\R_H,\delta,0$);
\EndProcedure
\end{algorithmic}
\end{method}
At  recursion $k$ in Algorithm~\ref{alg:recursion}, the number of solutions is multiplied by $\multiset{r_H(k)}{j_k}$.  Implementing this operation at each $k$ inflates the number of unique solutions to the number of ILP$_H$~(\ref{ILP}) solutions.  This is the number of ways of choosing the number of duplicates of each element appearing  within the US solution from the number of duplicates available. Hence, Method~\ref{meth:USP} is designed to compute
\[
\sum_{\substack{0\le i \le m-1 \\[1pt] 0\le j_i\le \lfloor \frac{\delta}{q_H(i)} \rfloor \\ j_0q_{H}(0)+\cdots+j_{m-1}q_{H}(m-1)=\delta}}\prod_{i=0}^{m-1} \multiset{r_H(i)}{j_i}.
\]

\section{A method for determining the number of decimation classes} \label{sec:CountExam}

{\blue The following lemma is used to develop a method for finding the number of decimation classes.  Its proof is identical to that found in~\cite{decimation}.}

\begin{lma}\label{lma:HuHv}
Let $G\cong\ZZZ$ have exponent $\ellstar$ and  $\uuu,\vv\in \Z_{\ge 0}^{G}$ be two vectors in the same decimation class. Let $H_{\uuu}, H_{\vv} \leq \mathbb{Z}^\times_{\ellstar}$ 
be the multiplier groups of $\uuu$ and $\vv$, respectively.
Then $H_{\uuu}=H_{\vv}$.
\end{lma}

\begin{comment}
\begin{proof}
Let $I_{\uuu}$ and $I_{\vv}$ be defined as in Definition~\ref{dfn:mul_vect}. Since $\uuu$ and $\vv$ are in the same decimation class, there exists constants $a \in \mathbb{Z}^\times_{\ellstar}$ and
 $b \in G$ such that $I_{\vv}=aI_{\uuu}+b$. Let $t \in H_{\vv}$.
Then $tI_{\vv}= I_{\vv}+g$ for some $g \in G$.
This implies that $I_{\vv}+g=atI_{\uuu}+tb$. Hence, 
$aI_{\uuu}+b+g=atI_{\uuu}+tb$. Then 
$$tI_{\uuu}=a^{\phi(\ellstar)-1}(aI_{\uuu}+b+g-tb)=I_{\uuu}+a^{\phi(\ellstar)-1}(b+g-tb).$$
So $H_{\vv}\leq H_{\uuu}$. Since $I_{\uuu}=a^{\phi(\ellstar)-1}I_{\vv}-a^{\phi(\ellstar)-1}b$, the same argument proves $H_{\uuu}\leq H_{\vv}$.
\end{proof}
\end{comment}

{\blue The method for finding the number of decimation classes from \cite{decimation} now carries over for nonnegative integer vectors.  This method requires generating the lattice of all subgroups of a finite abelian group.  The following lemma was proven in \cite{decimation} and aids in generating the lattice of all subgroups of an abelian group.}

\begin{lma}\label{lma:SNF} 
Let $K=\langle t_1,\dots,t_k\rangle$ be a finite abelian group, where $k$ is the smallest integer such that 
{$K=\langle x_1,\dots,x_k\rangle$ for some $x_1,\ldots, x_k \in G$.} Then each subgroup $J$ of $K$ can be generated by no more than $k$ elements. 
 \end{lma} 
 
 \begin{comment}
 \begin{proof}
 Let $\langle y_1,\ldots,y_m\rangle=J\leq K$ and both $K$ and $J$ be written additively. Then there is an $m\times k$ integer matrix $\A$  such that $\A\ttt=\y$, where 
 $\ttt=(t_1,\ldots,t_k)^\top$ and $\y=(y_1,\ldots,y_m)^\top$.  
Moreover, there exists $m \times m$, $k\times k$ and $m \times k$ integer matrices $\mP$, $\mQ$, $\D$ such that $\A=\mP\D\mQ$, 
$$\D= \begin{bmatrix}
d_1 & 0 & 0& & \dots & &0 \\
0 & d_2 & 0& &\dots  & &0\\
0 & 0 & \ddots& &\dots & &0\\
 \vdots &  & & d_r& & &\vdots  \\
   &  & & & 0& & \\
    &  & & & &\ddots & \\  
 0 & 0& 0&  & \dots & & 0\end{bmatrix}
,$$ where $d_i\, |\, d_{i+1}$ for all $i \in \{1,\ldots,r-1\}$, $r \leq \min\{k,m\}$,  
and $\mP$ and $\mQ$ both have integer matrix inverses~\cite{Larsen1974}. Then $\mP\D\mQ\ttt=\y$ implies $\D\mQ\ttt=\mP^{-1}\y$. Set $\uuu=\mQ\ttt$ and $\vv=\mP^{-1}\y$, then $\D\uuu=\vv$,  
$K=\langle u_1,\dots,u_k\rangle$, and
$J=\langle v_1,\dots,v_m\rangle$. Now, the equation $\D\uuu=\vv$  implies that $J=\langle v_1,\dots,v_r\rangle$. 
\end{proof} 
\end{comment}

 For any two subgroups, $K<L$, $K$ is called a \textit{maximal subgroup} of $L$ if $K <L$ and there is no subgroup $J$  such that $K<J<L$.
\textit{Cyclic extension} is a method for finding  all  subgroups of a finite abelian group $L=\langle u_1,\dots,u_k\rangle$, and constructing the lattice of all subgroups $\mathbb{L}$.  The lattice $\mathbb{L}$ is presented as a graph, where each distinct subgroup labels a distinct vertex, and  an edge from $J$ to $K$ exists if and only if $J$ is a maximal subgroup of $K$.  Cyclic extension first generates all of $L$'s cyclic subgroups 
%by generating all of its  %maximal cyclic subgroups 
$\langle v_i\rangle$
for each $1\le i\le m.$  It then recursively combines them to generate all subgroups of $L$. 
Let $\alpha_1$ be the number of subgroups of $\langle v_1\rangle$, and 
for each $2\le j\le m$ 
let $\alpha_j$ be the number of subgroups of $\langle v_j\rangle$ each of which is not a subgroup of any of the subgroups $\langle v_1\rangle, \langle v_2\rangle,\ldots, \langle v_{j-1}\rangle$. 
 By Lemma~\ref{lma:SNF}, it suffices to combine cyclic subgroups of at most $k$ cyclic
 % maximal 
 subgroups of $L$. Cyclic extension computes 
 $$\sum_{1\le i_1 \le m} \alpha_{i_1} + \sum_{1\le i_1<i_2 \le m}
 \alpha_{i_1}\alpha_{i_2} + \cdots + \sum_{1\le i_1<\cdots<i_k \le m}
 \alpha_{i_1}\cdots\alpha_{i_k}$$ number of all such subgroups to construct  $\mathbb{L}$.
 Cyclic extension uses a breadth-first search to minimize the number of groups generated. 
\textit{Lattice insertion} is a method used within cyclic extension for iteratively constructing $\mathbb{L}$.  Let $\mathbb{L}_i$ be the interim lattice containing only $i$ subgroups of $L$.  
The next group generated, $L_{i+1}$, is compared against each subgroup $K$  in $\mathbb{L}_i$ with $|K|$ dividing  $|L_{i+1}|$ and $J$ in $\mathbb{L}_i$ such that $|J|$ is divisible by $|L_{i+1}|$.  
If $K$ is a $\mathbb{L}_i$-maximal subgroup of $L_{i+1}$,
 i.e., there is no subgroup $K'$  in $\mathbb{L}_i$ 
such that $K<K'<L_{i+1}$,  
then an edge is added from $L_{i+1}$ to $K$.  Moreover, each edge from 
$J$ to $K$ in $\mathbb{L}_i$ such that $K<L_{i+1}<J$ is replaced with an edge from $J$ to $L_{i+1}$.    
  Each interim lattice $\mathbb{L}_i$ is guaranteed to be connected in this method as $\mathbb{L}_0$ contains only the trivial subgroup, $\langle 1 \rangle$, and at each step of the method, each newly introduced subgroup $L_{i+1}$ is either equal to a subgroup in $\mathbb{L}_i$ and discarded or one of the subgroups in $\mathbb{L}_i$ is an  $\mathbb{L}_i$-maximal subgroup of $L_{i+1}$. 
  
Method~\ref{meth:basic}
 (Count) determines the number of decimation classes of density $\delta$   nonnegative integer vectors indexed by a finite abelian group, of odd order $\ell$ and exponent $\ellstar$, such that $\gcd(\ellstar,\delta)=1$. 
\begin{method}[t]
\caption{A method for counting decimation classes of density $\delta$  nonnegative integer vectors indexed by a finite abelian group $G \cong \ZZZ$ such that $G$ is of odd order $\ell$, has exponent $\ellstar$, with $\gcd(\ellstar,\delta)=1$\label{meth:basic}}%\label{alg:count}
\begin{algorithmic}[1]
\Procedure{Count}{$\ell,\ellstar,\delta$}
\State {\bf Find} the set of all subgroups 
$\mathcal{H}$ of $\mathbb{Z}^\times_{\ellstar}$ along with $g_1,\ldots, g_{\gamma(H)}$ such that

\hspace{-.3cm}
 $\langle g_1,\ldots, g_{\gamma(H)}\rangle=H$, where $\gamma(H)$ is preferably small for each $H \in \mathcal{H}$;\label{step:start}
\State {\bf Construct} $\mathbb{L}$, the  lattice of subgroups for $\mathbb{Z}_{\ellstar}^\times$;\label{step:construct}
\State {\bf Set}  
$N'_{\langle1\rangle}=\frac{\multiset{\ell}{\delta}}{\ell}$\label{step:initial};
\For {$\langle g_1,\ldots, g_{\gamma(H)}\rangle=H\in \mathcal{H}\setminus \{\langle 1\rangle\}$}
\State {\bf Generate} all $H$-orbits;
\State {\bf Apply}  Method~\ref{meth:USP} to find the number of solutions $nsol$ to ILP$_H$~(\ref{ILP});\label{step:ILP}
\State {\bf Set} $N'_H:=\frac{nsol}{\prod_{1\le i\le r} \gcd(C,\ell_i)}$, where $C=\gcd(g_1-1,g_2-1,\ldots,g_{\gamma(H)}-1,\ellstar)$ to be 

\hspace{.3cm} the number of necklaces with  multiplier  groups  containing $H$; \label{step:multiplier}  % 
 \State {\bf Discount} $N'_H$, by  the number of necklaces  whose multiplier groups strictly contain
 
 \hspace{.3cm} $H$ from the top down in $\mathbb{L}$; \label{step:discount2}
  \State {\bf Set} the resulting value from Step~\ref{step:discount2} to be $N_H$, i.e., the number of necklaces whose 
  
 \hspace{.3cm}  multiplier groups are equal to $H$;

\State {\bf Set} $numD_H:=\frac{(N_H)*|H|}{\phi(\ellstar)}$;\label{step:laststatement}
\EndFor
 \State {\bf Set} $N_{\langle1\rangle}=N'_{\langle1\rangle}-\sum_{H\in \mathcal{H}\setminus \{\langle 1\rangle \}} N_H$;\label{step:N1}
  \State {\bf Set} $numD_{\langle 1\rangle} =
  \frac{N_{\langle1\rangle}}{\phi(\ellstar)}$;\label{step:numD1}\\
\Return $\sum_{H\in \mathcal{H}} numD_H$;\label{step:sum}
\EndProcedure
\end{algorithmic}
\end{method}
 In Step 2, 
 we used cyclic extension for generating the lattice of all subgroups $\mathbb{L}$ of $\mathbb{Z}^\times_{\ellstar}$ for $\ell=\ellstar$ and $3 \le \ellstar \le 121$.  While generating $\mathbb{L}$, all  subgroup members are stored to improve efficiency.  For example, if the generators of subgroup $K$ exist within subgroup $H$ and $|K|<|H|$, then $K<H$.  This also improved efficiency of subsequently generating the $H$-orbits by multiplying elements of $G=\mathbb{Z}_{\ell}$.  The sizes and duplicities of these orbits are recorded in $\R_H$ and $\Q_H$, respectively, to solve the corresponding US.
Step~\ref{step:initial} 
computes the number of
necklaces containing $\langle 1\rangle$ based on Lemma~\ref{lem:necklace}.
Step~\ref{step:multiplier} 
is necessary by Theorem~\ref{thm:OneFixAll-dan} for obtaining the number of necklaces containing $H$ from the number of solutions to ILP$_H$~(\ref{ILP}). 
Let $\mathcal{H}$ be the set of all subgroups of $H$. In Step~\ref{step:discount2}, our method of counting necklaces with a given multiplier group, $H \in \mathcal{H} \setminus \{ \langle 1 \rangle\}$, discounts the number of necklaces whose multiplier groups strictly contain  $H$ from that of necklaces whose multiplier groups contain $H$.  The resulting count is the number of necklaces with multiplier group $H$.  
Steps~\ref{step:laststatement}~and~\ref{step:numD1} reduce necklace counts from Steps~\ref{step:discount2}~and~\ref{step:N1} to decimation counts by applying the last statement of Lemmas~\ref{lma:stab2} and~\ref{lma:HuHv}. Step~\ref{step:sum} sums counts across all multiplier groups to get  the total number of decimation classes.

\end{document}